\documentclass[twoside, 12pt]{amsart}

\usepackage[utf8]{inputenc}
\usepackage{blindtext}
\usepackage[T1]{fontenc}
\usepackage{amsmath}

\usepackage{amssymb}
\usepackage{MnSymbol}
\usepackage{amsthm}
\usepackage{graphicx}
\usepackage{color,epsfig}      
\usepackage{tikz}
\usetikzlibrary{arrows,shapes}

\newtheorem{Lemma}{Lemma}
\newtheorem{Proposition}{Proposition}
\newtheorem{Theorem}{Theorem}

\newtheorem{Definition}{Definition}
\newtheorem{Problem}{Problem}

\newtheorem{Corollary}{Corollary}
\newtheorem{Remark}{Remark}

\numberwithin{Subcase}{Case}

\renewcommand{\Re}{\mathbb R}

\newcommand{\HH}{\mathcal{H}}

\newcommand{\KK}{\mathcal{K}}

\renewcommand{\S}{\mathbb{S}}

\DeclareMathOperator{\conv}{conv}
\DeclareMathOperator{\perim}{perim}

\DeclareMathOperator{\area}{area}
\DeclareMathOperator{\bd}{bd}

\DeclareMathOperator{\inter}{int}

\newcommand{\bs}{\bar{s}}
\newcommand{\bS}{\bar{S}}
\newcommand{\bt}{\bar{t}}
\newcommand{\bT}{\bar{T}}

\begin{document}

\title[On a Dowker-type problem]{On a Dowker-type problem for convex disks with almost constant curvature}

\author[B. Basit]{Bushra Basit}
\author[Z. L\'angi]{Zsolt L\'angi}

\address{Bushra Basit, Department of Algebra and Geometry, Budapest University of Technology and Economics\\
M\H uegyetem rkp. 3., H-1111 Budapest, Hungary}
\email{bushrabasit18@gmail.com}
\address{Zsolt L\'angi, Department of Algebra and Geometry, Budapest University of Technology and Economics, and MTA-BME Morphodynamics Research Group\\
M\H uegyetem rkp. 3., H-1111 Budapest, Hungary}
\email{zlangi@math.bme.hu}

\thanks{Partially supported by the National Research, Development and Innovation Office, NKFI, K-147544 grant}



\begin{abstract}
A classical result of Dowker (Bull. Amer. Math. Soc. 50: 120-122, 1944) states that for any plane convex body $K$, the areas of the maximum (resp. minimum) area convex $n$-gons inscribed (resp. circumscribed) in $K$ is a concave (resp. convex) sequence. It is known that this theorem remains true if we replace area by perimeter, or convex $n$-gons by disk-$n$-gons, obtained as the intersection of $n$ closed Euclidean unit disks. It has been proved recently that if $C$ is the unit disk of a normed plane, then the same properties hold for the area of $C$-$n$-gons circumscribed about a $C$-convex disk $K$ and for the perimeters of $C$-$n$-gons inscribed or circumscribed about a $C$-convex disk $K$, but for a typical origin-symmetric convex disk $C$ with respect to Hausdorff distance, there is a $C$-convex disk $K$ such that the sequence of the areas of the maximum area $C$-$n$-gons inscribed in $K$ is not concave. The aim of this paper is to investigate this question if we replace the topology induced by Hausdorff distance with a topology induced by the surface area measure of the boundary of $C$.
\end{abstract}

\maketitle

\section{Introduction}\label{sec:intro}

In the paper, by a convex disk we mean a compact, convex set with nonempty interior in the Euclidean plane $\Re^2$, and denote the closed, Euclidean unit disk centered at the origin $o$ by $B^2$. Furthermore, we denote by $\KK$ and $\KK_o$ the families of convex disks and origin-symmetric convex disks, respectively. The families of the elements of $\KK$ and $\KK_o$ with $C^2$-class boundary and strictly positive curvature are denoted by $\KK_{2+}$ and $\KK_{o,2+}$, respectively. 

For any integer $n \geq 3$ and plane convex body $K$, let $A_n(K)$ (resp. $a_n(K)$) denote the the infimum (resp. supremum) of the areas of the convex polygons with at most $n$ vertices circumscribed about (resp. inscribed in) $K$. A famous result of Dowker \cite{Dowker} states that for any convex disk $K$, the sequences $\{ A_n(K) \}$ and $\{ a_n(K) \}$ are convex and concave, respectively. It was proved independently by L. Fejes T\'oth \cite{LFTperim}, Moln\'ar \cite{Molnar} and Eggleston \cite{Eggleston} that the same statements remain true if we replace area by perimeter, where the last author also showed that these statements are false if we replace area by Hausdorff distance. These results are known to be true also in any normed plane \cite{MSW}. Dowker's theorems and their variants have become important in many areas of discrete geometry, in particular in the theory of packing and covering \cite{LFTSzeged, Bambah, regfig} and are often used (see e.g. \cite{Prosanov, BL23}).

We mention only one variant of these theorems, for which we first
need some definitions. If $C$ is an origin-symmetric convex disk in the Euclidean plane $\Re^2$, we call a set $X$ \emph{$C$-spindle convex}, or shortly \emph{$C$-convex}, if it is $\Re^2$, or there is a translate of $C$ with $X \subseteq C$ and for any $p,q \in X$, $X$ contains the intersection of all translates of $C$ containing $\{ p,q\}$. This concept was introduced by Mayer \cite{Mayer} in 1935, and in modern times its systematic investigation, for the special case that $C$ is the Euclidean closed unit disk $B^2$ centered at the origin $o$, began in \cite{BLNP}, and the general case for example in \cite{LNT2013}. A result related to spindle convex sets is due to G. Fejes T\'oth and Fodor \cite{TF2015} who extended Dowker's theorems, together with their variants for perimeter, for spindle convex sets; in these theorems the role of inscribed or circumscribed convex $n$-gons is played by the so-called \emph{disk-$n$-gons}, obtained as the intersections of $n$ closed Euclidean unit disks. They also proved similar theorems for spindle convex sets in hyperbolic or spherical plane, except for the perimeter of circumscribed disk-polygons in spherical plane.

These results were generalized in \cite{BL}, where the authors proved that for any $o$-symmetric convex disk $C$, the same statements hold for the areas of circumscribed $C$-$n$-gons, and the $C$-perimeters of inscribed and circumscribed $C$-$n$-gons in any $C$-convex disk $K$, where a $C$-$n$-gon is defined to be the intersection of $n$ translates of $C$. In addition, they showed that for a typical $C \in \KK_o$, with respect to the topology induced by Hausdorff distance, there is a $C$-convex disk $K$ such that the sequence of the areas of the maximum area $C$-$n$-gons inscribed in $K$ is not concave.
Here we recall that the elements of a residual set of a topological space $\tau$ are called \emph{typical}, where a residual set is a countable intersection of sets of dense interiors in $\tau$.

In the following we say that $C \in \KK_o$ satisfies the \emph{Dowker Property} if the above sequence is concave for any $C$-convex disk $K$.
We remark that the proof of almost all Dowker-type theorems mentioned in the introduction are based on different versions of the same inequality for convex (resp. $C$-convex) quadrangles; we are going to define this inequality for the areas of inscribed $C$-$n$-gons in Section~\ref{sec:prelim}, and say that an $o$-symmetric convex disk satisfying this inequality for any $C$-convex quadrangle satisfies the \emph{Quadrangle Property}, we note that any $C \in \KK_o$ satisfying the Quadrangle Property also satisfies the Dowker Property (for a more precise and detailed discussion, see Section~\ref{sec:prelim}).

Our first result is a new proof of the theorem in \cite{TF2015} stating that the Euclidean unit disk $B^2$ satisfies the Dowker Property. We prove Theorem~\ref{thm:old} by showing that $B^2$ satisfies the Quadrangle Property. In order to state it, for any $n \geq 3$, $C \in \KK_o^H$, and $C$-convex disk $K$, we set $\hat{a}_n^C(K) = \sup \{ \area(Q) : Q \hbox{ is a \textit{C}-\textit{n}-gon inscribed in }  K \}$.

\begin{Theorem}\label{thm:old}
For any $n \geq 4$ and spindle convex disk $K$, we have $\hat{a}_{n-1}^{B^2}(K) + \hat{a}_{n+1}^{B^2}(K) \leq 2 \hat{a}_n^{B^2}(K)$.
\end{Theorem}

Note that by the results in \cite{BL}, the result of  G. Fejes T\'oth and Fodor \cite{TF2015} about the areas of disk-$n$-gons inscribed in a spindle convex disk $K$ cannot be extended for `slightly perturbed' circles in place of $C$. In other words, if $\KK_o^H$ denotes the topological space induced by Hausdorff distance on $\KK_o$, $B^2$ has no neighborhood $U$ in $\KK_o^H$ such that any $C \in U$ satisfies the Dowker Property.
On the other hand, one may observe that if the boundary of some $C \in \KK_o^H$ has `almost unit curvature' at every point, then $C$ is `close' to $B^2$ with respect to Hausdorff distance, but the converse does not hold even if $C$ is assumed to have $C^2$-class boundary with strictly positive curvature.
Thus, it is a natural question to ask if any $C \in \KK_o^H$ with almost constant curvature satisfies the Dowker Property or not.

In order to investigate this question, we introduce a new topology on the family $\KK_o$. This topology is induced by a generalized distance function $d_{PM} : \KK_o \times \KK_o \to \Re \cup \{ \infty \}$, where $d_{PM}(C_1,C_2)$ is defined by the surface area measures of $C_1$ and $C_2$.
We call this generalized distance \emph{perimeter measure distance} or \emph{PM-distance}, and denote the topological space induced by it on $\KK_o$ by $\KK_o^{PM}$.
We show that $\KK_o^{PM}$ is a refinement of $\KK_o^H$.
We observe also that if $\KK_{o,2+}^{PM}$ denotes the subfamily of $\KK_o^{PM}$ consisting of the elements having $C^2$-class boundary and strictly positive curvature, then the property that $C_1 \in \KK_{o,2+}^{PM}$ is `close' to $C_2 \in \KK_{o,2+}^{PM}$ is equivalent to the property that the curvatures of $\bd(C_1)$ and $\bd(C_2)$ at points with the same outer normal vector are `almost equal'; in particular, within this class of convex disks, PM-distance coincides with the $L_{\infty}$ norm of the radius of curvature as a function of the outer unit vector.

Our second result is the following.

\begin{Theorem}\label{thm:counter}
A typical element $C \in \KK_{o,2+}^{PM}$ does not satisfy the Quadrangle Property.
\end{Theorem}

We ask the following.
 
\begin{Problem}\label{prob:counter}
Prove or disprove that a typical element $C \in \KK_{o}^{PM}$ does not satisfy the Dowker Property. What happens if we replace $\KK_{o}^{PM}$ by $\KK_{o,2+}^{PM}$, or the Dowker Property by the Quadrangle Property?
\end{Problem}

The structure of the paper is as follows. In Section~\ref{sec:prelim}, we give precise definitions of the concepts mentioned in the introduction, investigate their properties, and prove some lemmas necessary to prove our theorems. Finally, in Sections~\ref{sec:old} and \ref{sec:counter}, we prove Theorems~\ref{thm:old} and \ref{thm:counter}, respectively.

In the paper, for any $p,q \in \Re^2$, we denote by $[p,q]$ the closed segment with endpoints $p,q$, by $||p-q||$ the Euclidean distance of $p$ and $q$, and by $|p,q|$ the value of the $2 \times 2$ determinant whose columns are $p$ and $q$.

\section{Preliminaries}\label{sec:prelim}

We start with a formal definition of a $C$-spindle.

\begin{Definition}\label{defn:C-spindle}
Let $C \in \KK_o$, and consider two (not necessarily distinct) points $p, q \in \Re^2$ such that a translate of $C$ contains both $p$ and $q$.
Then the \emph{$C$-spindle} of $p$ and $q$, denoted by $[p,q]_C$, is the intersection of all translates of $C$
that contain $p$ and $q$. If no translate of $C$ contains $p$ and $q$, we set $[p,q]_C = \Re^2$.
We call a set $K \in \KK$ \emph{$C$-spindle convex},  or shortly \emph{$C$-convex}, if for any $p,q \in K$, we have $[p,q]_C \subset K$.
\end{Definition}

We recall from \cite[Corollary 3.13]{LNT2013} that a closed set in $\Re^2$, different from $\Re^2$,  is $C$-convex if and only if it is the intersection of some translates of $C$.
Furthermore, we recall that for any set $X \subseteq \Re^2$, the \emph{$C$-convex hull}, or shortly \emph{$C$-hull} is the intersection of all $C$-convex sets that contain $X$. We denote it by $\conv_C(X)$, and note that it is $C$-convex, and if $X$ is compact, then it coincides with the intersection of all translates of $C$ containing $X$ \cite{LNT2013}.

\begin{Definition}\label{defn:Cpolygon}
The intersection of $n$ translates of $C$ is called a $C$-$n$-gon.
\end{Definition}

Note that any $C$-$n$-gon is the $C$-hull of at most $n$ points, and vice versa \cite{BL}.

Let $K$ denote a $C$-convex disk, and for any distinct points $p,q \in \bd(K)$, let $\widehat{pq}$ denote the connected arc of $\bd(K)$ such that we move along $\bd (K)$ according to the positive orientation of the plane defined by the standard basis vectors. Furthermore, let $r_K^C(p,q)$ denote the region bounded by $\widehat{pq}$ and the $C$-arc in the boundary of $[p,q]_C$ that runs from $p$ to $q$ in counterclockwise direction; in this notation we may omit the symbol $C$ if it is clear which $o$-symmetric convex disk it refers to.

For any $p \in \bd(K)$, a point $q \in \bd(K)$ is \emph{antipodal} to $p$ if $K$ has parallel supporting lines at $p$ and $q$. If the arc $\widehat{pq} \subset \bd(K)$ contains no point antipodal to $p$, we say that \emph{$\widehat{pq}$ has turning angle strictly less than $\pi$}.

\begin{Definition}\label{defn:Qproperty}
Let $C \in \KK_o$. Assume that for any $C$-convex disk $K$, and for any boundary points $x_1,x_2,x_3,x_4$ of $K$ in this counterclockwise order, if $\widehat{x_1x_4}$ has turning angle strictly less than $\pi$, then
\[
\area(r_K(x_1,x_4))+\area(r_K(x_2,x_3)) \geq \area(r_K(x_1,x_3))+ \area(r_K(x_2,x_4)).
\]
Then we say that $C$ satisfies the \emph{Quadrangle Property} (see Figure~\ref{fig:quadrangle}).
\end{Definition}

\begin{figure}[ht]
\begin{center}
 \includegraphics[width=0.4\textwidth]{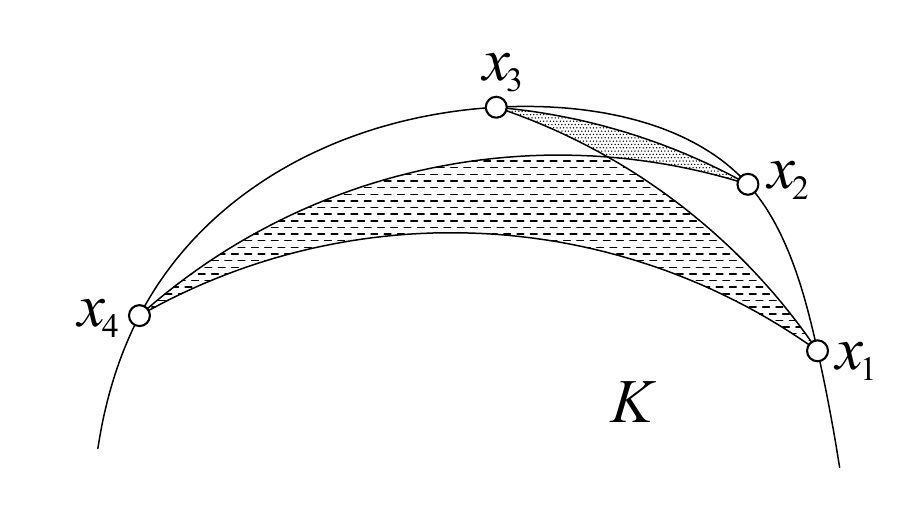}
 \caption{An illustration for the Quadrangle Property. According to it, the area of the dashed region is not less than the area of the dotted region.}
\label{fig:quadrangle}
\end{center}
\end{figure}

\begin{Definition}\label{defn:Dproperty}
Let $C \in \KK_o$. Assume that for any $C$-convex disk $K$ and $n \geq 4$, we have
\[
\hat{a}_{n-1}^C(K) + \hat{a}_{n+1}^C(K) \leq 2 \hat{a}_n^C(K).
\]
Then we say that $C$ satisfies the \emph{Dowker Property}.
\end{Definition}

Our next remark follows by a straightforward modification of \cite[Lemma 4]{BL}. The same statement for inscribed convex polygons or disk-polygons is the main idea of the proofs in \cite{Dowker} and \cite{TF2015}, respectively.

\begin{Remark}\label{rem:QtoD}
If $C \in \KK_o$ satisfies the Quadrangle Property, then it satisfies the Dowker Property.
\end{Remark}

Our next lemma can be proved by a straightforward modification of the proof of \cite[Lemma 5]{BL}.

\begin{Lemma}\label{lem:derivative}
Let $C \in \KK_o$, and let $K$ be a $C$-convex disk. Let $\Gamma(s)$ be the arclength-parametrization of the boundary of $K$.
Then, for any $x_i=\Gamma(s_i)$ with $i=1,2,3,4$ in this counterclockwise order where $\widehat{x_1x_4}$ has turning angle strictly less than $\pi$, we have
\[
\area(r_K(x_1,x_4))+\area(r_K(x_2,x_3)) \geq  \area(r_K(x_1,x_3))+ \area(r_K(x_2,x_4))
\]
if and only if $\left. \partial_s \partial_t \area(r_K(\Gamma(s),\Gamma(t)))\right|_{(s,t)=(\bs,\bt)} \geq 0$ for all $\Gamma(\bs)=p \neq \Gamma(\bt)=q$ where $\widehat{pq}$ has turning angle strictly less than $\pi$.
\end{Lemma}

In the remaining part of this section we introduce a new metric on $\KK_o$ and examine its properties.
For any convex disk $K$, and for any $X \in \S^1$, let $N_K(X)$ denote the set of the boundary points $p$ of $X$ with the property that an element of $X$ is an outer unit normal vector of $K$ at $p$. Then, for any Borel set $\sigma \subseteq \S^1$, the \emph{surface area measure} of $\sigma$, denoted by $A_K(\sigma)$ is defined as the Hausdorff measure of $N_K(\sigma)$. We note that if $K$ has a $C^2$-class boundary with positive curvature, then for any Borel set $\sigma \subseteq \S^1$, $A_K(\sigma)$ is the integral of the radius of curvature of $K$ at $p(u)$ over $u \in \sigma$, where $p(u)$ denotes the unique point of $\bd(K)$ with outer unit normal vector $u$ \cite{Schneider}.

For any Borel set $\sigma \subset \S^1$, let $\HH(\sigma)$ 
denote the Hausdorff measure of $\sigma$.

\begin{Definition}\label{eq:PM_distance}
Let $C,D$ be $o$-symmetric convex bodies. We define the \emph{perimeter measure distance} or \emph{PM-distance} of $C$ and $D$ as
\begin{multline}\label{eq:PM}
d_{PM}(C,D)= \\
\inf \left\{ \tau \geq 0 : |A_C(\sigma) - A_D(\sigma)| \leq \tau \cdot \HH(\sigma) \hbox{ for any Borel-set } \sigma \subseteq \S^1  \right\}
\end{multline}
if the infimum exists, and otherwise we set $d_{PM}(C,D)=\infty$.
\end{Definition}

Note that by its definition, $d_{PM}(\cdot,\cdot)$ is clearly a nonnegative, symmetric function satisfying the triangle inequality. Furthermore, by \cite[Theorem 8.1.1]{Schneider}, $d_{PM}(C,D)=0$ and $C,D \in \KK_o$ imply that $C = D$.

 As we have remarked in the introduction, the function $d_{PM}(\cdot,\cdot)$ induces a topology on $\KK_o$, and the induced topological space is denoted by $\KK_o^{PM}$. It is worth noting that this topology can be naturally extended to the family $\KK$, or to families of convex bodies in $d$-dimensional Euclidean space for any value of $d$.

We note that by its definition, within the class $\KK_{o,2+}$, the statement of Proposition~\ref{prop:refinement} is obvious. We also note that the topology on $\KK_o$ defined by Hausdorff distance can also be generated via surface area measures. This was done in \cite{HS} by Hug and Schneider, who proved, in particular, that the topology induced by the Hausdorff distance of compact, convex sets is equivalent to the one induced by the bounded Lipschitz distance of their surface area measures.

\begin{Proposition}\label{prop:refinement}
Let $\{ C_n \}$ be a sequence of convex disks in $\KK_o$ and let $C \in \KK_o$. If  $C_n \to C$ with respect to the topology induced by $d_{PM}(\cdot,\cdot)$, then $C_n \to C$ with respect to the topology induced by $d_H(\cdot,\cdot)$. Furthermore, there is a sequence $\{ D_n \}$ of convex disks in $\KK_o$ and a convex disk $D \in \KK_o$ such that $D_n \to D$ with respect to $d_H(\cdot,\cdot)$, but $D_n \not\to D$ with respect to $d_{PM}(\cdot,\cdot)$.
\end{Proposition}

\begin{proof}
Assume that $C_n \to C$ with respect to $d_{PM}(\cdot,\cdot)$, and let $\varepsilon > 0$ be a sufficiently small fixed value. Then, if $n$ is sufficiently large, for any Borel set $\sigma \in \S^1$, we have
\[
|A_{C_n}(\sigma) - A_C(\sigma)| \leq \varepsilon \cdot \HH(\sigma) \leq 2 \pi \varepsilon.
\]
In particular, we obtain that if $n$ is sufficiently large, then $|\perim(C_n)-\perim(C)| \leq 2 \pi \varepsilon$, implying that for some $0 < r < R$, each $C_n$ and $C$ contains $r B^2$ and is contained in $R B^2$. Thus, by \cite[Theorem 8.5.1]{Schneider}, we have that if $\varepsilon$ is sufficiently small, then, for some translate $C_n'$ of $C_n$, we have $d_H(C_n',C) \leq \gamma \varepsilon^{1/2}$ for all sufficiently large values of $n$, where $\gamma$ depends only on $r$ and $R$. On the other hand, since $C_n$ and $C$ are $o$-symmetric, this implies $d_H(C_n,C) < \gamma \varepsilon^{1/2}$, and hence, $C_n \to C$ with respect to Hausdorff distance.

Now, let $D = B^2$, and $D_n = \conv (K \cup \{p_n, -p_n \})$, where $p_n=\left( 1+\frac{1}{n}, 0  \right)$. Then, clearly, $D_n \to D$ with respect to Hausdorff distance. On the other hand, setting $\sigma =  \{ (0,1) \}$, we have that $A_{D_n}(\sigma) > 0$, and $A_D(\sigma) = \HH(\sigma) = 0$, showing that $d_{PM}(D_n,D) = \infty$ for all values of $n$.
\end{proof}

\begin{Corollary}\label{cor:refinement}
The topology on $\KK_o$ induced by $d_{PM}(\cdot,\cdot)$ is a refinement of the topology induced by $d_H(\cdot,\cdot)$.
\end{Corollary}

Based on the proof of Proposition~\ref{prop:refinement}, we remark that if $D_n \to D$ with respect to $d_{PM}(\cdot,\cdot)$, then we have that for sufficiently large values of $D$, the family of the outer unit normal vectors of the smooth points of $D_n$ coincides with the same set for $D$.
Furthermore, we note that Proposition~\ref{prop:refinement} indicates that a set in $\KK_o$ which is residual with respect to $d_H(\cdot,\cdot)$ is not necessarily residual in the topology induced by $d_{PM}(\cdot,\cdot)$.

\begin{Remark}\label{rem:isolated}
Let $K \in \KK_o$ contain a segment $[p,q]$ on its boundary. Assume that $L \in \KK_o$ satisfies $d_{PM}(K,L)< \infty$. Then, from the definition of $d_{PM}(\cdot,\cdot)$, we obtain that $L$ contains a translate of $[p,q]$ on its boundary of the same length. In particular, it follows that if $P$ and $Q$ are two polygons with $d_{PM}(P,Q) < \infty$, then $Q$ is a translate of $P$. Thus, the topology induced by $d_{PM}(\cdot,\cdot)$ on the family of $o$-symmetric convex polygons is the discrete topology.
\end{Remark}


\section{Proof of Theorem~\ref{thm:old}}\label{sec:old}

Let $C \in \KK_{o,2+}$, and let $K$ be a $C$-convex disk. Without loss of generality, we assume that the curvature of a point of $\bd(K)$ is strictly greater than the curvature of $\bd(C)$ at the point with the same outer unit normal.

\begin{Remark}\label{rem:longestchord}
By our conditions about the curvatures at the corresponding points of $\bd (K)$ and $\bd (C)$, it is easy to see that for any points $p,q \in \bd (K)$, the set $[p,q]_C \setminus \{ p,q \}$ is contained in $\inter (K)$.
\end{Remark}

Let us parametrize $\bd(K)$ and $\bd(C)$ as the curves $\Gamma : [0,\perim(K)] \to \Re^2$ and $\Theta : [0,\perim(C)] \to \Re^2$, respectively, using arclength-parametrization and running along the boundaries of $C$ and $K$ in positive orientation. We choose the starting points $\Gamma(0)$ and $\Theta(0)$ in such a way that here $\bd(K)$ and $\bd(C)$ have the same outer unit normal.
For any two points $p=\Gamma(s_1), q=\Gamma(s_2)$, let us denote the arc $\widehat{pq}$ of $\Gamma$ by $\Gamma|_{[s_1,s_2]}$, and
set $A(s_1,s_2) = \area(r_K(\Gamma(s_1),\Gamma(s_2)))$.

Our main tool will be the following. First, we compute the mixed partial derivative $(\partial_s \partial_t A)(\bar{s},\bar{t})$ for any points $p=\Gamma(\bar{s}), q=\Gamma(\bar{t})$,  where $\widehat{pq}$ has turning angle strictly less than $\pi$.
We note that our formula for $(\partial_s \partial_t A)(\bar{s},\bar{t})$ will be valid for any $C \in \KK_{o,2+}$, not only for $B^2$.
In the second part, we show that if $C=B^2$, then this derivative is nonnegative, and apply Lemma~\ref{lem:derivative}.

\subsection{Deriving a formula for the mixed partial derivative $(\partial_s \partial_t A)(\bar{s},\bar{t})$}\label{subsec:formula}

Without loss of generality, let $p=\Theta(\bar{S}), q=\Theta(\bar{T})$, where we assume that $0 < \bar{s} < \bar{t} < \perim(K)$ and $0 < \bar{S} < \bar{T} < \perim(C)$.

Note that by Remark~\ref{rem:longestchord}, for any $0 < s < t < \perim(K)$, there are unique values $0 < S' < T' < \perim(C)$ with the property that $\Theta(T')-\Theta(S') = \Gamma(t)-\Gamma(s)$. Thus, if we set $H = \{ (s,t) : 0 < s < t < \perim(K) \}$, we may define the functions $S, T : H \to \Re$ to identically satisfy $\Theta(T(s,t))-\Theta(S(s,t)) = \Gamma(t)-\Gamma(s)$. Note that $\bar{S}=S(\bar{s},\bar{t})$ and $\bar{T}=T(\bar{s},\bar{t})$. More specifically, consider the function $f(s,t,S',T')= \Theta(S')-\Theta(T')+\Gamma(s)-\Gamma(t)$. This is a twice continuously differentiable function in a neighborhood of $(\bs, \bt, \bS, \bT)$, and satisfies $f(\bs,\bt,\bS,\bT)=0$. Its Jacobian is the $2 \times 4$ matrix $[\Gamma'(s) \quad -\Gamma'(t) | -\Theta'(S')  \quad \Theta'(T')]$. Since $\Theta'(\bS)$ and $\Theta'(\bT)$ are not parallel, the matrix $[-\Theta(S') \quad \Theta'(T')]$ is invertible. Thus, by the Implicit Function Theorem, the functions $S(s,t)$ and $T(s,t)$ exist and are twice continuously differentiable in a neighborhood of $(\bs,\bt)$.

Our next remark will play an important role in the proof.

\begin{Remark}\label{rem:order}
Since $K$ is $C$-convex, and $[p,q]_C$ is contained in $\inter(K)$ apart from $p,q$, we immediately have that the following vectors are in this counterclockwise order around $o$: $p-q$, $- \Theta'(\bT)$, $- \Gamma'(\bt)$, $\Gamma'(\bs)$, $\Theta'(\bS)$, $q-p$, $\Theta'(\bT)$, $\Gamma'(\bt)$, $-\Gamma'(\bs)$, $-\Theta'(\bS)$, $p-q$.
\end{Remark}

\begin{figure}[ht]
\begin{center}
 \includegraphics[width=0.4\textwidth]{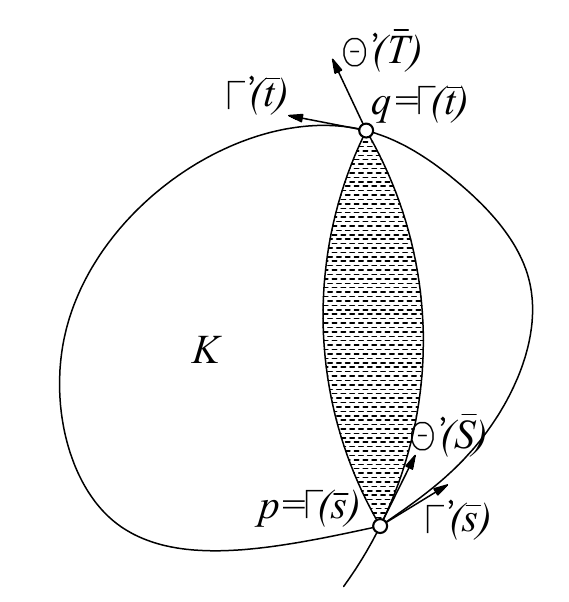}
 \caption{The vectors described in Remark~\ref{rem:order}. The $C$-spindle $[p,q]_C$ is illustrated by a dashed region.}
\label{fig:directions}
\end{center}
\end{figure}

We denote the oriented angles of the vectors listed in Remark~\ref{rem:order}, measured from the positive half of the $x$-axis, by $\alpha, \theta_q-\pi, \gamma_p, \theta_p, \alpha+\pi, \theta_q, \gamma_q, \theta_p+\pi, \alpha+2\pi$, respectively. 

\begin{Lemma}\label{lem:partials}
Let $\kappa_p, \kappa_q$ denote the curvatures of $\bd(C)$ at $p$ and $q$, respectively. Then
\[
(\partial_s S)(\bs, \bt) = \frac{\sin (\theta_q - \gamma_p)}{\sin(\theta_q-\theta_p)}, \, (\partial_t S)(\bs, \bt) = \frac{\sin(\gamma_q-\theta_q)}{\sin(\theta_q-\theta_p)},
\]
\[
(\partial_s T)(\bs, \bt) = \frac{\sin (\theta_p - \gamma_p)}{\sin(\theta_q-\theta_p)}, \, (\partial_t T)(\bs, \bt) = \frac{\sin(\gamma_q-\theta_p)}{\sin(\theta_q-\theta_p)}, \hbox{ and}
\]
\begin{multline*}
(\partial_s \partial_t S)(\bs, \bt) = \\
= \frac{\kappa_p \sin (\gamma_q-\theta_q) \cos (\theta_q-\theta_p) \sin(\theta_q-\gamma_p) - \kappa_q \sin (\theta_p - \gamma_p) \sin(\gamma_q-\theta_p)}{\sin^3 (\theta_q - \theta_p)}.
\end{multline*}
\end{Lemma}

\begin{proof}
By the Implicit Function Theorem, in a neighborhood of $(\bs, \bt)$ we have
\begin{equation}\label{eq:extra}
\left[ \begin{array}{cc}
\partial_s S & \partial_t S\\
\partial_s T & \partial_t T
\end{array} \right] =
-[ -\Theta'(S) \quad \Theta'(T) ]^{-1} [ \Gamma'(s) \quad -\Gamma'(t) ].
\end{equation}
For $* \in \{ S,T \}$, let us denote the $x$ and $y$ coordinates of $\Theta'(*)$ with $\Theta'_x(*)$ and $\Theta'_y(*)$, respectively.
Then
\[
[ -\Theta'(S) \quad \Theta'(T) ]^{-1} = \frac{1}{| \Theta'(S) , \Theta'(T)| }
\left[ \begin{array}{cc}
\Theta'_y(T) & -\Theta'_x(T)\\
\Theta'_y(S) & -\Theta'_x(S)
\end{array} \right] .
\]
Computing the product of the two matrices in (\ref{eq:extra}), we obtain that $\partial_s S = \frac{| \Gamma'(s)  , \Theta'(T)|}{| \Theta'(S)  , \Theta'(T)|}$. Since $\Gamma$ and $\Theta$ are parametrized by arclength, the vectors appearing in this expression are unit vectors. Since for any vectors $u,v \in \Re^2$, the determinant $|u , v |$ is the signed area of the parallelogram spanned by them, we obtain that $| \Gamma'(\bs)  , \Theta'(\bT)| = \sin(\theta_q-\gamma_p)$ and $| \Theta'(\bS)  , \Theta'(\bT)| = \sin (\theta_q-\theta_p)$. We obtain the other three first partial derivatives in a similar way; in particular we have $\partial_t S= \frac{|\Theta'(T), \Gamma'(t)|}{| \Theta'(S) , \Theta'(T)|}$. In order to compute $(\partial_s \partial_t S)(\bs, \bt)$, we differentiate this expression, and obtain that

\begin{multline}\label{eq:dsdtS}
\partial_s \partial_t S= \frac{1}{{| \Theta'(S) , \Theta'(T)|^2}} \bigl( | \Theta''(T), \Gamma'(t)| (\partial_s T) | \Theta'(S) , \Theta'(T)| - \\
- |\Theta'(T) , \Gamma'(t)| \left( | \Theta''(S) , \Theta'(T)| (\partial_s S) + | \Theta'(S) , \Theta''(T)| (\partial_s T) \right) \bigr).
\end{multline}

Note that by arclength-parametrization, we have
\begin{equation}\label{eq:curvature}
\Theta''(\bS) = \kappa_p \left[ \begin{array}{c} \cos\left(\theta_p + \frac{\pi}{2} \right)\\ \sin\left(\theta_p + \frac{\pi}{2} \right) \end{array} \right], \quad
\Theta''(\bT) = \kappa_q \left[ \begin{array}{c} \cos\left(\theta_q + \frac{\pi}{2} \right)\\ \sin\left(\theta_q + \frac{\pi}{2} \right) \end{array} \right].
\end{equation}
Thus, in order to derive the expression for $(\partial_s \partial_t S)(\bs, \bt)$ in Lemma~\ref{lem:partials}, we substitute these expressions and our formulas for the first partial derivatives of $S$ and $T$ into (\ref{eq:dsdtS}), and use trigonometric identities.
\end{proof}

Now, let $\bar{A}(s,t)= \area(K)-A(s,t)$, and we prove that $(\partial_s \partial_t \bar{A})(\bs,\bt) \geq 0$.
Let $(s,t)$ be sufficiently close to $(\bs,\bt)$. We compute $\bar{A}(s,t)-\bar{A}(\bs,\bt)$.

Note that for any $s,t$, there is a unique vector $x(s,t)$ such that  $x(s,t)+C$ contains the oriented $C$-arc from $\Gamma(s)$ to $\Gamma(t)$.
In order to compute this vector, we observe that $\Gamma(s)=x(s,t) + \Theta(S(s,t))$ and $\Gamma(t)=x(s,t) + \Theta(T(s,t))$, i.e.
\begin{equation}\label{eq:x}
x(s,t)= \Gamma(s) - \Theta(S(s,t)) = \Gamma(t) - \Theta(T(s,t)).
\end{equation}
By the properties of $S(s,t)$ and $T(s,t)$, the function $x(s,t)$ has continuous second partial derivatives.

Let $R$ denote the region swept by the $C$-arc from $p$ to $q$ as it is translated in the direction of $x$ to its translate by $x$ (see Figure~\ref{fig:areas}). We denote the signed area of $R$ by $A(R)$, where the sign is defined such that the part of $R$ inside $K$ has negative area, and the part outside it has positive area.

Let $X_p$ denote the closed oriented curve obtained by joining the oriented segment from $p$ to $x+p$, the arc $x+\Theta([\bS, S(s,t)])$ oriented from $x+p$ to $\Gamma(s)$, and the arc $\Gamma([\bs,s])$ oriented from $\Gamma(s)$ to $p$. 
Similarly, let $X_q$ denote the closed oriented curve obtained by joining the oriented segment from $q$ to $x+q$, the arc $x+\Theta([\bT, T(s,t)])$ oriented from $x+q$ to $\Gamma(t)$, and the arc $\Gamma([\bt,t])$ oriented from $\Gamma(t)$ to $q$. Let $R_p$ and $R_q$ denote the region enclosed by $X_p$ and $X_q$, respectively. For $* \in \{ p,q \}$, we say that the signed area $A(R_*)$ of the region $R_*$ is positive or negative if $X_*$ runs in counterclockwise or clockwise direction on its boundary. Then
\begin{equation}\label{eq:areadiff}
\bar{A}(s,t)-\bar{A}(\bs,\bt)= A(R) -A(R_p) + A(R_q).
\end{equation}

\begin{figure}[ht]
\begin{center}
 \includegraphics[width=0.43\textwidth]{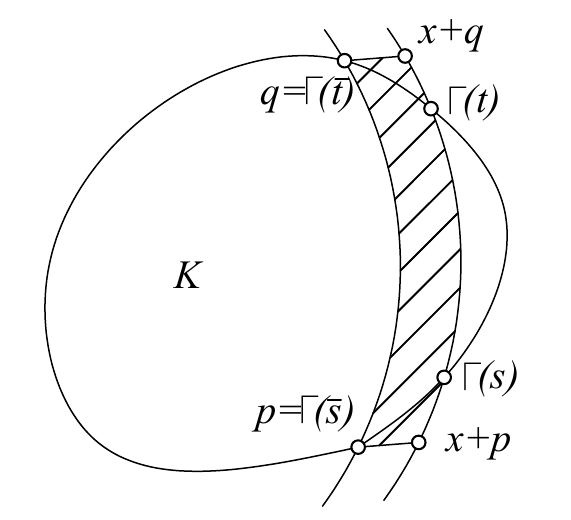}
 \caption{An illustration for the notation in the proof of Theorem~\ref{thm:old}. The region $R$ is denoted by skew lines.}
\label{fig:areas}
\end{center}
\end{figure}

In the following part we compute the second mixed derivatives of $A(R)$, $A(R_p)$ and $A(R_q)$.

First, we compute the value of $\partial_s \partial_t A(R)$ at $(\bs,\bt)$.
By Fubini's theorem, $A(R)$ is equal to the signed area of the parallelogram spanned by $x$ and $q-p$, or more precisely, $A(R) = |x,q-p|=|p-q,x|$.
Since $p-q$ is a constant, we have $\partial_s \partial_t A(R) = | p-q,\partial_s \partial_t x |$. Now, by (\ref{eq:x}), we have
\[
\partial_s \partial_t x = \partial_s \left( - \Theta'(S) (\partial_t S) \right) = - \Theta''(S) (\partial_s S)(\partial_t S) - \Theta'(S) (\partial_s\partial_t S).
\]
The formulas in (\ref{eq:curvature}) yield that
\[
\frac{|p-q,(\partial_s \partial_t x)(\bs,\bt)|}{||p-q||} =  - \kappa_p \sin \left( \theta_p + \frac{\pi}{2} - \alpha \right) (\partial_s S)(\partial_t S) - \sin(\theta_q-\alpha) (\partial_s\partial_t S).
\]
Thus, by Lemma~\ref{lem:partials} and using trigonometric identities, we obtain that
\begin{multline}\label{eq:dsdtx}
\frac{|p-q,(\partial_s \partial_t x)(\bs,\bt)|}{||p-q||} = \frac{1}{\sin^3 (\theta_q-\theta_p)} \left( 
\kappa_p \sin(\alpha-\theta_q) \sin(\theta_q-\gamma_p) \sin (\gamma_q-\theta_q) + \right. \\
\left.  + \kappa_q \sin (\theta_p-\alpha) \sin(\theta_p-\gamma_p) \sin(\gamma_q-\theta_p) \right) .
\end{multline} 
Note that by Remark~\ref{rem:order}, all angles in (\ref{eq:dsdtx}) are convex angles. This and the inequalities $\kappa_p, \kappa_q, ||p-q|| > 0$ imply that $|p-q,(\partial_s \partial_t x)(\bs,\bt)| > 0$.

Next, we compute the value of $\partial_s \partial_t A(R_p)$ at $(\bs,\bt)$. 
Recall the well-known formula that if $G : [s_1,s_2] \to \Re^2$ is an arclength-parametrized curve enclosing a region and running on its boundary in counterclockwise direction, then the area of the enclosed region is $\frac{1}{2} \int_{s_1}^{s_2} | G(s), G'(s)| \, ds$. Using this fact, we obtain that 
\begin{equation}\label{eq:ARp}
2 A(R_p)= \int_{\bS}^{S(s,t)} |x(s,t) + \Theta(u) - \Gamma(\bs), \Theta'(u) | \, du + \int_{s}^{\bs} | \Gamma(u)-\Gamma(\bs), \Gamma'(u)| \, du. 
\end{equation}
Note that $\partial_t \int_{s}^{\bs} | \Gamma(u)-\Gamma(\bs), \Gamma'(u)| \, du = 0$. Furthermore, $S(s,t)$ has continuous second partial derivatives, and the function $\partial_t |x + \Theta(u) - \Gamma(\bs), \Theta'(u) |$ is continuous in an open neighborhood of the set $\{ (\bs,u) : u \in [\bS, S(s,t)] \}$. Thus, we may apply Leibniz's integral rule and obtain that
\[
\partial_t 2A(R_p) = | x(s,t) + \Theta(S(s,t)) - \Gamma(\bs), \Theta'(S(s,t)) | (\partial_t S) + \int_{\bS}^{S(s,t)} | (\partial_t x), \Theta'(u) | \, du
\]
Again applying Leibniz's integral rule and the identities $x(\bs,\bt)=0$, $\Theta(\bS)=\Gamma(\bs)$, we have 
\begin{multline*}
\left. \partial_s \partial_t 2A(R_p) \right|_{(s,t)=(\bs,\bt)} = \\
= | (\partial_s x)(\bs,\bt) + \Theta'(\bS) (\partial_s S)(\bs,\bt), \Theta'(\bS) | (\partial_t S)(\bs,\bt)
+ | 0, \Theta''(\bS)| (\partial_s S)(\bs,\bt) (\partial_t S)(\bs,\bt)+ \\
+ | 0, \Theta'(\bS) | (\partial_s \partial_t S)(\bs,\bt) + | (\partial_t x)(\bs, \bt), \Theta'(\bS)| (\partial_s S)(\bs,\bt) + \int_{\bS}^{\bS} | (\partial_s \partial_t x)(\bs,\bt), \Theta'(u)| \, du.
\end{multline*} 
This yields that
\[
\left. \partial_s \partial_t 2A(R_p) \right|_{(s,t)=(\bs,\bt)} = | (\partial_s x)(\bs,\bt), \Theta'(\bS) | (\partial_t S)(\bs,\bt) + | (\partial_t x)(\bs, \bt), \Theta'(\bS)| (\partial_s S)(\bs,\bt).
\]
 Now (\ref{eq:x}) yields $(\partial_s x)(\bs,\bt) = - \Theta'(\bT) (\partial_s T)(\bs,\bt)$ and $(\partial_t x)(\bs,\bt) = - \Theta'(\bS) (\partial_t S)(\bs,\bt)$.
Substituting these into the above expression yields that
\[
\left. \partial_s \partial_t A(R_p) \right|_{(s,t)=(\bs,\bt)} = - \frac{1}{2} | \Theta'(\bT), \Theta'(\bS)| (\partial_s T)(\bs,\bt) (\partial_t S)(\bs,\bt).
\]
Thus, applying Lemma~\ref{lem:partials}, we have
\begin{equation}\label{eq:dsdtArp}
\left. \partial_s \partial_t A(R_p) \right|_{(s,t)=(\bs,\bt)} = \frac{\sin(\gamma_q-\theta_q) \sin (\theta_p-\gamma_p)}{2 \sin (\theta_q-\theta_p)}
\end{equation}

Finally, using a similar argument, we obtain that
\begin{equation}\label{eq:ARq}
2 A(R_q)= \int_{\bT}^{T(s,t)} |x(s,t) + \Theta(u) - \Gamma(\bt), \Theta'(u) | \, du + \int_{t}^{\bt} | \Gamma(u)-\Gamma(\bt), \Gamma'(u)| \, du, \hbox{ and}
\end{equation}
\begin{equation}\label{eq:dsdtArq}
\left. \partial_s \partial_t A(R_q) \right|_{(s,t)=(\bs,\bt)} = -\frac{\sin(\gamma_q-\theta_q) \sin (\theta_p-\gamma_p)}{2 \sin (\theta_q-\theta_p)}
\end{equation}
Thus,
\[
\left. \partial_s \partial_t A(R_q)-A(R_p) \right|_{(s,t)=(\bs,\bt)} = -\frac{\sin(\gamma_q-\theta_q) \sin (\theta_p-\gamma_p)}{\sin (\theta_q-\theta_p)},
\]
implying that
\begin{multline}\label{eq:genform}
(\partial_s \partial_t A)(\bar{s},\bar{t}) = \\
= -\frac{\sin(\gamma_q-\theta_q) \sin (\theta_p-\gamma_p)}{\sin (\theta_q-\theta_p)} + \frac{||p-q||}{\sin^3 (\theta_q-\theta_p)} \left(  \kappa_p \sin(\alpha-\theta_q) \sin(\theta_q-\gamma_p) \cdot \right. \\
\left. \cdot \sin (\gamma_q-\theta_q) + \kappa_q \sin (\theta_p-\alpha) \sin(\theta_p-\gamma_p) \sin(\gamma_q-\theta_p) \right) .
\end{multline}

\subsection{Determining the sign of the derivative if $C$ is a Euclidean disk}\label{subsec:Euclidean}

In the following, we assume that $C$ is a Euclidean disk of radius $r$, and compute the value of the expression in (\ref{eq:genform}).
Using the notation in this expression, we can assume that $\theta_p=\pi-\theta$ and $\theta_q=\pi+\theta$ for some $0 < \theta < \frac{\pi}{2}$, which, by Remark~\ref{rem:order}, implies that $\theta < \gamma_p < \pi - \theta$ and $\pi + \theta < \gamma_q < \gamma_p + \pi$. Moreover, we have $||p-q|| \kappa_ p= ||p-q|| \kappa_q =  2 \sin \theta$. Substituting these into (\ref{eq:genform}), we obtain 
\begin{multline}\label{eq:genform2}
(\partial_s \partial_t A)(\bar{s},\bar{t}) =
\frac{ 2 \sin^2{\theta} \left( \sin(\gamma_p-\theta) \sin (\theta-\gamma_q) - \sin(\theta + \gamma_p) \sin(\gamma_q+\theta) \right) }{\sin^3 (2 \theta)} - \\ -\frac{\sin(\theta- \gamma_q) \sin(\theta + \gamma_p)}{\sin(2 \theta)}.
\end{multline} 

Using the trigonometric identity $\sin^2 2\theta = 4 \sin^2 \theta \cos^2 \theta$, we obtain that
\begin{multline}\label{eq:genform3}
2 \cos^2 \theta \sin(2 \theta) (\partial_s \partial_t A)(\bar{s},\bar{t}) =
\sin(\gamma_p-\theta) \sin (\theta-\gamma_q) - \\
- \sin(\theta + \gamma_p) \sin(\gamma_q+\theta) - 2\cos^2 \theta \sin(\theta- \gamma_q) \sin(\theta + \gamma_p).
\end{multline}
By additional trigonometric identities, this expression simplifies to
\begin{multline}\label{eq:genform4}
2 \cos^2 \theta \sin(2 \theta) (\partial_s \partial_t A)(\bar{s},\bar{t}) = \\
= - \sin^2 \theta  \cos(\gamma_p + \gamma_q) - \cos(\gamma_p - \gamma_q)  + \cos^2 \theta \cos(2 \theta + \gamma_p - \gamma_q).
\end{multline}

Now, a simple computation shows that under our conditions and considering $\gamma_p - \gamma_q$ fixed, the quantity in (\ref{eq:genform4}) is minimal if and only if $\gamma_p+\gamma_q = 2\pi$, implying that
\begin{multline}\label{eq:genform5}
2 \cos^2 \theta \sin(2 \theta) (\partial_s \partial_t A)(\bar{s},\bar{t}) \geq \\
\geq - \sin^2 \theta - \cos(2 \gamma_p)  + \cos^2 \theta \cos(2 \theta + 2 \gamma_p).
\end{multline}
Note that $\gamma_p+\gamma_q = 2\pi$ also yields that $\frac{\pi}{2} < \gamma_p$. If we set $\gamma = \gamma_p - \frac{\pi}{2}$, then the right-hand side expression in (\ref{eq:genform5}) can be written as
\begin{multline*}
- \sin^2 \theta + \cos(2 \gamma)  - \cos^2 \theta \cos(2 \theta + 2 \gamma) = \\
= 2 \cos (\gamma + \theta) \left( \cos (\gamma - \theta) - \cos^2 \theta \cos (\gamma + \theta) \right),
\end{multline*}
where our conditions for the angles are transformed into the inequalities  $0 < \theta, \gamma$ and $\gamma + \theta < \frac{\pi}{2}$.
Thus, we have that $\cos (\gamma + \theta) > 0$, and also that $\cos (\gamma - \theta) > \cos (\gamma + \theta) > \cos^2 \theta \cos (\gamma + \theta)$.
This yields Theorem~\ref{thm:old}.

\section{Proof of Theorem~\ref{thm:counter}}\label{sec:counter}

Let $\KK_-$ denote the family of $o$-symmetric convex disks in $\KK_{o,2+}$ for which there is a $C$-convex disk $K$ that does not satisfy the Quadrangle Property. We show that $\KK_-$ is residual in $\KK_{o,2+}^{PM}$.

First, we show that $\KK_-$ is open in $\KK_{0,2+}^{H}$, which, by Proposition~\ref{prop:refinement}, yields that it is open in $\KK_{0,2+}^{PM}$.
Let $C \in \KK_-$, and suppose for contradiction that there is a sequence $C_m \notin \KK_-$ with $C_m \to C$.
Since $C \in \KK_-$, there is a $C$-convex disk $K$ and points $x_1,x_2,x_3,x_4$ in this counterclockwise order in $\bd(K)$, with turning angle strictly less than $\pi$, such that $\area(r_K^C(x_1,x_4))+\area(r_K^C(x_2,x_3)) < \area(r_K^C(x_1,x_3))+ \area(r_K^C(x_2,x_4))$.
As $K$ is $C$-convex, it is the intersection of some translates of $C$, i.e. there is some $X \subset \Re^2$ such that $K = \bigcap_{x \in X} (x+C)$.
Let $K_m = \bigcap_{x \in X} (x+C_m)$. Then, clearly, $K_m$ is $C_m$-convex, and $K_m \to K$. For $i=1,2,3,4$, let $x_i^m \in \bd(K_m)$ be a point with $x_i^m \to x_i$.
Note that if $m$ is sufficiently large, $x_1^m, x_2^m, x_3^m, x_4^m$ are in this counterclockwise order in $\bd(K^m)$, and the turning angle of $\widehat{x_1^mx_4^m}$ is strictly less than $\pi$. Thus, $\area(r_{K_m}^{C_m}(x_1^m,x_4^m))+\area(r_{K_m}^{C_m}(x_2^m,x_3^m)) \geq \area(r_{K_m}^{C_m}(x_1^m,x_3^m))+ \area(r_{K_m}^{C_m}(x_2^m,x_4^m))$. To the contrary, \cite[Remark 2]{BL} implies that $\area(r_K^C(x_1,x_4))+\area(r_K^C(x_2,x_3)) \geq \area(r_K^C(x_1,x_3))+ \area(r_K^C(x_2,x_4))$.

In the remaining part of the proof, we show that $\KK_-$ is everywhere dense in $\KK_{o,2+}^{PM}$.
Note that for any $C \in \KK_{o,2+}^{PM}$ and Borel set $\sigma \subset \S^1$, $A_C(\sigma) = \int R_C(u) \, d \sigma$, where $R_C(u)$ denotes the radius of curvature at the point of $\bd(C)$ with outer unit normal $u$. In particular, if for $C, C' \in \KK_{o,2+}$ we have that for any $u \in \S^1$, $|R_C(u)-R_{C'}(u)| \leq \varepsilon$, then $d_{PM}(C,C') \leq \varepsilon$.

Based on this, we observe that for any $\lambda_0 > 0$, the subfamily of $\KK_{o,2+}$ of the elements $C$ satisfying the property that there is a point $p \in \bd(C)$ where
\begin{itemize}
\item the curvature of $\bd(C)$ is locally maximal,
\item if $\Theta(s)$ is the arclength-parametrization of $\bd(C)$ with $p=\Theta(0)$, then $\Theta$ is a $C^{\infty}$-class function in a neighborhood of $0$,
\item If $\kappa(s)$ denotes the curvature of $\Theta(s)$, then $\kappa(s) = \kappa_0 - \lambda s^2 + O(s^3)$ for some $\kappa_0 > 0$ and any given $\lambda > \lambda_0$,
\end{itemize}
is everywhere dense in $\KK_{o,2+}^{PM}$.
Indeed, let $\Theta_0(s)$ be the arclength-pa\-ram\-e\-triza\-tion of the boundary of an arbitrary element $C_0$ of $\KK_{o,2+}$. Without loss of generality, assume that the curvature $\kappa_0(s)$ of $\Theta_0(s)$ has a local maximum at $s=0$. 
Let $\varepsilon > 0$ and $\lambda > 0$ be fixed, and choose some $\delta > 0$ such that $|\kappa_0(s)-\kappa_0(0)| \leq \frac{\varepsilon}{2}$ for any $|s| \leq \delta$. Replace the function $\kappa_0(s)$ by the function $\kappa(s) = \kappa_0(0)-\lambda s^2$ in a small neighborhood $[-s_0,s_0]$ of $0$ in $(-\delta,\delta)$. We extend this function to the interval $[-\delta,\delta]$ in such a way that the obtained function $\kappa(s)$ is continuous and strictly positive, satisfies $|\kappa_0(s) - \kappa(s)| \leq \frac{\varepsilon}{2}$ for any $s \in [-s_0,s_0]$, and for $\bar{s} \in \{ -s_0,s_0 \}$, it satisfies
\begin{itemize}
\item[(i)] $\int_0^{\bar{s}} \kappa_0(t) \, dt = \int_0^{\bar{s}} \kappa(t) \, dt$;
\item[(ii)] $\int_0^{\bar{s}} \cos \int_0^t \kappa_0(u) \, du \, dt = \int_0^{\bar{s}} \cos \int_0^t \kappa(u) \, du \, dt$; and
\item[(iii)] $\int_0^{\bar{s}} \sin \int_0^t \kappa_0(u) \, du \, dt = \int_0^{\bar{s}} \sin\int_0^t \kappa(u) \, du \, dt$.
\end{itemize}
We observe that then $\kappa(s)$ is the curvature function of a convex disk in $\KK_{2+}$. Modifying also a neighborhood of the point $-\Gamma_0(0)$ in the same way yields an $o$-symmetric convex disk $C$ at PM-distance at most $\varepsilon$ from $C_0$ with a curvature function $\kappa(s)=\kappa_0-\lambda s^2$ in a neighborhood of $s=0$. 

Thus, in the following we assume that the above conditions hold for $C$ with $\lambda > \frac{9 \kappa_0^3}{2}$.
Let $\Theta(0)=(0,0)$ such that $\Theta'(0)$ points in the direction of the negative half of the $x$-axis.
As before, let $\kappa(s)$ denote the curvature of the curve $\Theta$ at the point $\Theta(s)$. We set $p=\Theta(-\bar{s})$ and $q=\Theta(\bar{s})$ for some sufficiently small value of $\bar{s}$, and set $\gamma_p = \frac{\pi}{2}+\kappa_0 \bs$ and $\gamma_q = \frac{3\pi}{2}-\kappa_0 \bs$. Note that then $\theta_p = \pi-\kappa_0 \bs +\frac{\lambda}{3} \bs^3$ and $\theta_q = \pi+\kappa_0 \bs -\frac{\lambda}{3} \bs^3$.

Then, computing the coordinates of $p$ and $q$ as in (ii) and (iii), we obtain 
\[
||p-q|| = 2\bs-\frac{\kappa_0^2}{3}\bs^3+ \frac{\kappa_0(8\lambda+\kappa_0^3)}{60} \bs^5 + O(\bs^7),
\]
and, from $p_y=q_y$ we obtain that $p-q$ points in the direction of the positive half of the $x$-axis. Substituting these into (\ref{eq:genform}), we obtain that
\[
(\partial_s \partial_t A)(-\bar{s},\bar{s}) = \frac{ 4\kappa_0 (9 \kappa_0^3 - 2 \lambda)}{3}  \bs^4 + O(\bs^6).
\]
This shows that the mixed partial derivative $(\partial_s \partial_t A)(-\bar{s},\bar{s})$ is negative if $\bs > 0$ is sufficiently small, implying that $C$ does not satisfy the Quadrangle Property.

\section*{Acknowledgements}

The research published in this paper has been partially supported by the National Research, Development and Innovation Office, NKFI, K-147544 grant.

The authors express their gratitude to an unknown referee for many helpful suggestions, and for directing their attention to the connection between Hausdorff distance and bounded Lipschitz metric.


\begin{thebibliography}{}

\bibitem{Bambah} R.P. Bambah and C.A. Rogers, \emph{Covering the plane with convex sets}, J. London Math. Soc. \textbf{27} (1952), 304-314.

\bibitem{BL} B. Basit and Z. L\'angi, \emph{Dowker-type theorems for disk-polygons in normed planes},  	arXiv:2307.04026 [math.MG], DOI:10.48550/arXiv.2307.04026


\bibitem{BL23} K. Bezdek and Z. Lángi, \emph{From the separable Tammes problem to extremal distributions of great circles in the unit sphere}, Discrete Comput. Geom.
(2023), DOI:10.1007/s00454-023-00509-w

\bibitem{BLNP} K. Bezdek, Z. Lángi, M. Naszódi and P. Papez, \emph{Ball-polyhedra}, Discrete Comput. Geom. \textbf{38} (2007), 201-230.


\bibitem{Dowker} C.H. Dowker, \emph{On minimum circumscribed polygons}, Bull. Amer. Math. Soc. \textbf{50} (1944), 120-122.

\bibitem{Eggleston} H.G. Eggleston, \emph{Approximation to plane convex curves. (I) Dowker-type theorems}, Proc. London Math. Soc. (3) \textbf{7} (1957), 351-377.

\bibitem{GFT} G. Fejes T\'oth, \emph{On a Dowker-type theorem of Eggleston}, Acta Math. Sci. Hungar. \textbf{29} (1977), 131-148.


\bibitem{TF2015} G. Fejes T\'oth and F. Fodor, \emph{Dowker-type theorems for hyperconvex discs}, Period. Math. Hungar. \textbf{70} (2015), 131-144.

\bibitem{LFTSzeged} L. Fejes T\'oth, \emph{Some packing and covering theorems}, Acta Sci. Math. (Szeged) \textbf{12}/A (1950), 62-67.

\bibitem{LFTperim} L. Fejes T\'oth, \emph{Remarks on polygon theorems of Dowker}, Mat. Lapok \textbf{6} (1955), 176-179 (Hungarian).

\bibitem{regfig} L. Fejes T\'oth, \emph{Regular Figures}, Macmillan, New York, 1964.

\bibitem{HS} D. Hug and R. Schneider, \emph{H\"{o}lder continuity for support measures of convex bodies}, Arch. Math. (Basel) \textbf{104} (2015), 83-92.

\bibitem{LNT2013} Z. L\'angi, M. Naszo\'di and I. Talata, \emph{Ball and Spindle Convexity with respect to a Convex Body}, Aequationes Math. \textbf{85} (2013), 41-67.


\bibitem{MSW} H. Martini, K. Swanepoel and G. Weiss, \emph{The geometry of Minkowski spaces - a survey. Part I}, Expo. Math. \textbf{19} (2001), 97-142 .

\bibitem{Mayer} A.E. Mayer, \emph{Eine \"Uberkonvexit\"at}, Math. Z. \textbf{39} (1935), 511-531.

\bibitem{Molnar} J. Moln\'ar, \emph{On inscribed and circumscribed polygons of convex regions}, Mat. Lapok \textbf{6} (1955), 210-218 (Hungarian).

\bibitem{Prosanov} R. Prosanov, \emph{On a relation between packing and covering densities of convex bodies}, Discrete Comput. Geom. \textbf{65} (2021), 1028–1037.

\bibitem{Schneider} R. Schneider, \emph{Convex Bodies: The Brunn-Minkowski Theory}, Encyclopedia of Mathematics and Its Applications \textbf{154}, second extended edition, Cambridge University Press, Cambridge, UK, 2014. 


\end{thebibliography}
\end{document}